\documentclass[12pt, oneside, a4paper]{article}

\usepackage{amsmath}
\usepackage{amsfonts}
\usepackage{amssymb}
\usepackage{amsthm,mathrsfs}
\usepackage{enumerate}
\usepackage{graphicx}
\usepackage{subfigure}
\newtheorem{theorem}{Theorem}[section]
\newtheorem{lemma}[theorem]{Lemma}

\theoremstyle{definition}
\newtheorem{remark}[theorem]{Remark}
\newtheorem{example}[theorem]{Example}

\title{\textbf{A new class of bi-transversal matroids}}
\author{Mahdi Ebrahimi\footnote{ m.ebrahimi.math@ipm.ir \& m.ebrahimi.math@gmail.com}
 \\
 Declarations of interest: none\\
 {\small\em  School of Mathematics, Institute for Research in Fundamental Sciences (IPM)},\\{\small\em P.O. Box: 19395-5746, Tehran, Iran}\footnote{This research was supported by a grant from IPM.}}

\date{}

\begin{document}

\maketitle

\begin{abstract}
A transversal matroid whose dual is also transversal is called bi-transversal. Let $G$ be an undirected graph with vertex set $V$. In this paper, for every subset $W$ of $V$, we associate a bi-transversal matroid to the pair $(G,W)$. We also derive an explicit formula for counting bases of this matroid.

 \end{abstract}

\section{Introduction}
$\indent$ The terminology used here for matroids will in general follow \cite{oxley}. For a matroid $M$ with ground set $E$, the dual of $M$ and the deletion of $X\subseteq E$ from $M$ are denoted by $M^\ast$ and $M\backslash X$, respectively. Transversal matroids are an important class of matroids discovered by edmonds and fulkerson \cite{A}. Transversal matroids that are also co-transversal are called \textit{bi-transversal} \cite{bi-transversal}. The characterization of bi-transversal matroids is an interesting problem originally posed by Welsh \cite{welsh}.

 Las Vergans \cite{vergans} proved that every principal transversal matroid is bi-transversal. The class of lattice path matroids defined by Stanley \cite{stanley} is another example of bi-transversal matroids \cite{lattice}. In this paper, we wish to present a new class of bi-transversal matroids. For this purpose, we require some concepts from graph theory.

Let $G=(V,E)$ be an undirected graph with vertex set $V$ and edge set $E$. The degree $\mathrm{deg}_G(v)$ of a vertex $v \in V$ is the number of edges that are incident to the vertex $v$; when $G$ is a multi-graph, a loop contributes $2$ to a vertex's degree, for the two ends of the edge. Let $\infty$ be a symbol with $\infty\notin V$.
 Now suppose $\alpha=(m(v),\,v\in V)$ be a vector of non-negative integers.
  \textit{A labeling of $G$ with respect to $\alpha$} is a map
$\varphi_\alpha$ from $E$ into $V\cup \{\infty\}$
satisfying:\\
a) For every $e\in E$ with ends $v$ and $w$, we have $\varphi_\alpha(e)\in \{v,w,\infty\}$, and\\
b) $|\varphi_\alpha^{-1}(v)|\leqslant m(v)$, for each $v\in V$.\\
We define the \textit{height} of $\varphi_\alpha$ as $\mathrm{Hei}(\varphi_\alpha):=|\{e\in E|\,\varphi_\alpha(e)\neq \infty\}|$.
For instance, in Figure \ref{labeling}, we can see  two labeling of graphs with respect to given vectors. In each case, the $i$th entry of the vector $\alpha$ or $\beta$ shows $m(i)$. Note that for the graph (1), $\mathrm{Hei}(\varphi_\alpha)=10$ and for the graph (2) $\mathrm{Hei}(\varphi_\beta)=6$.
\begin{figure}[th!]
\vspace{-0.2cm}
\begin{center}
\includegraphics[width=15cm]{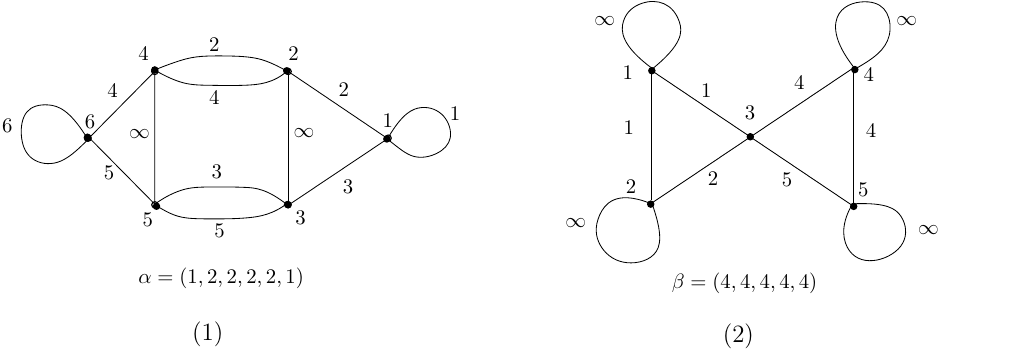}
\label{labeling}
\vspace{-0.4cm}
\caption{labeling of graphs.}
\end{center}
\end{figure}

Let $S:=\{(v,i)\in V\times \mathbb{N}|\, 1\leqslant i\leqslant \mathrm{deg}_G(v)\}$. For a subset $X$ of $S$, the
 \textit{$\alpha$-vector} of $X$ is the vector $\alpha(X):=(m_X(v),\,v\in V)$, where for every
  $v\in V$,
   $m_X(v):=|\{i|\,(v,i)\in X\}|$.
   Also we say that $X$ is a \textit{perfect subset} of $S$,
    if for some $W\subseteq V$, we have
    $X= \{(v,i)\in W\times \mathbb{N}|\, 1\leqslant i\leqslant \mathrm{deg}_G(v)\}$. In this case, we use the notation $S(W)$ for the perfect set $X$.
For example, for the graph (1) in Figure \ref{labeling}, the $\alpha$-vector of the set $X:=\{(1,1), (1,2), (1,3), (1,4), (3,1), (3,3), (4,4), (5,4), (6,1), (6,2), (6,3)\}$ is $\alpha(X)=(4,0,2,1,1,3)$. Now we are ready to state our main result.
\begin{theorem}\label{main}
Suppose $G=(V,E)$ is an undirected graph and $S:=\{(v,i)\in V\times \mathbb{N}|\, 1\leqslant i\leqslant \mathrm{deg}_G(v)\}$. Assume that $r$ is a function from $2^{S}$ into $\mathbb{Z}$ defined by $r(X):=max \{\mathrm{Hei}(\varphi)|\,\varphi\,\text{is a labeling of } G \text{ with respect to } \alpha(X)\}$.\\
\textbf{a)} $M=(S,r)$ is an identically self dual transversal matroid of rank $|E|$ with rank function $r$.\\
\textbf{b)} For every perfect subset $X$ of $S$, $M\backslash X$ is a by transversal matroid.
\end{theorem}

\begin{remark}
Let $G=(V,E)$ be an undirected graph it is clear that the matroid described in Theorem \ref{main} (a) only depends on the structure of $G$. So we use the notation $\mathrm{TM}(G)$ for this matroid. Also we denote the ground set and the rank function of $\mathrm{TM}(G)$ by  $S_G$ and $r_G$, respectively. Now let $X\subseteq S_G$ be a perfect set. We use the notation $\mathrm{TM}(G,W)$ for the matroid $\mathrm{TM}(G)\backslash X$, where  $W$ is a subset of $ V$ such that $X=S_G(W)$.
\end{remark}

Assume that $M$ is a matroid. We say that $M$ is a \textit{graphical transversal matroid}, if there exist an undirected graph $G=(V,E)$ and a subset (possibly empty) $W$ of $V$ such that $M\cong \mathrm{TM}(G, W)$ or $M\cong \mathrm{TM}(G, W)^\ast$. In the sequel, we wish to derive an explicit formula for counting bases of graphical transversal matroids.

Let $G=(V,E)$ be an undirected graph with $V=\{v_1,v_2,\dots, v_n\}$. An \textit{orientation} $\mathcal{O}$ of $G$ is an assignment of a direction to each edge in $E$, turning the initial graph $G$ into a directed graph $\mathcal{O}(G)$. In $\mathcal{O}(G)$, a directed edge $e$ from $v$ to $w$ is called an \textit{arc} of the directed graph $\mathcal{O}(G)$ denoted by $\mathcal{O}(e)=(v,w)$. For the arc $\mathcal{O}(e)=(v,w)$, the vertex $w$ is called the \textit{head} $h(\mathcal{O}(e))$ and $v$ is called the \textit{tail}  $t(\mathcal{O}(e))$ of the arc $\mathcal{O}(e)$. For a vertex $v\in V$, the number of arcs with head $v$ is called the \textit{in-degree} $\mathrm{deg}_\mathcal{O}^-(v)$ of $v$ and the number of arcs with tail $v$, the \textit{out-degree} $\mathrm{deg}_\mathcal{O}^+(v)$ of $v$.
We define an equivalence relation $\sim$ on the set of orientations on $G$ by $\mathcal{O}_1\sim \mathcal{O}_2$, if and only if $(\mathrm{deg}_{\mathcal{O}_1}^+(v_1),\mathrm{deg}_{\mathcal{O}_1}^+(v_2),\dots,\mathrm{deg}_{\mathcal{O}_1}^+(v_n))=(\mathrm{deg}_{\mathcal{O}_2}^+(v_1),\mathrm{deg}_{\mathcal{O}_2}^+(v_2),\dots,\mathrm{deg}_{\mathcal{O}_2}^+(v_n))$. We use the notation $D^+(G)$ for the set of all equivalence classes of $\sim$. Also for every orientation $\mathcal{O}$ of $G$, we denote by $(\mathrm{deg}_{\mathcal{O}}^+(v_1),\mathrm{deg}_{\mathcal{O}}^+(v_2),\dots,\mathrm{deg}_{\mathcal{O}}^+(v_n))$, the equivalence class of $\sim$ containing $\mathcal{O}$ (See example \ref{exam}).

\begin{theorem}\label{graphical}
Suppose $G=(V,E)$ is an undirected graph and $W\subseteq V$ with $V\backslash W=\{v_1,v_2,\dots, v_n\}$. Then the number of bases of $\mathrm{TM}(G,W)$ is equal to
$$\sum_{(a_1,a_2,\dots,a_n)\in D^+(G[V\backslash W])}\prod_{i=1}^n\binom{\mathrm{deg}_G(v_i)}{a_i},$$
where $G[V\backslash W]$ is the induced subgraph of $G$ on $V\backslash W$.
\end{theorem}

\section{Graphical transversal matroids}
Recall that a set system $\mathcal{A}$ on a set $E$ is a multi-set of subsets of $E$. It is convenient to write $\mathcal{A}$ as $(A_1, A_2, \dots, A_n)$ or $(A_i,\, i\in [n]:=\{1,2,\dots,n\})$ with the understanding that $(A_{\tau(1)},A_{\tau(2)},\dots,A_{\tau(n)})$, where $\tau$ is any permutation of $[n]$, is the same set system. A \textit{partial transversal} of $\mathcal{A}$ is a subset $I$ of $E$ for which there is an injection $f:I\rightarrow [n]$ with $i\in A_{f(i)}$ for all $i\in I$.
 Edmonds and Fulkerson \cite{A} showed that the partial transversals of a set system $\mathcal{A}$ on $E$ are the independent sets of a matroid on $E$; we say that $\mathcal{A}$  is a \textit{presentation} of this \textit{transversal matroid} $M[\mathcal{A}]$.\\
 Assume that $M$ is a matroid with ground set $E$. For a subset $X\subseteq E$, the \textit{contraction} of $X$ from  $M$ is defined  by $M/X:=(M^\ast\backslash X)^\ast$.
 Now we are ready to prove our main results on the structure of graphical transversal matroids.\\

\noindent\textbf{Proof of Theorem \ref{main}: } For every $e\in E$ with ends $v$ and $w$, suppose $A(e):=\{(x,i)\in S|\,x\in \{v,w\},\, 1\leq i\leq \mathrm{deg}_G(x)\}$.\\
\textbf{ a)}
 Consider the set system $\mathcal{A}:=(A(e),\, e\in E)$ and let $r_{\mathcal{A}}$ be the rank function of the transversal matroid $M[\mathcal{A}]$.
  We first show that $M=M[\mathcal{A}]$ and $r=r_{\mathcal{A}}$.
   Assume that $X\subseteq S$ and  $B=\{b_1, b_2,\dots, b_k\}$ is a basis of $X$ in $M[\mathcal{A}]$.
   Since $B$ is a partial transversal, there exist distinct edges $e_1, e_2,\dots,e_k \in E$ so that for every $1\leq j\leq k$, we have $b_j\in A(e_j)$.
  Also for every $1\leq j\leq k$, there exist $v_j\in V$ and $1\leq i_j \leq \mathrm{deg}_G(v_j)$ such that $b_j =(v_j,i_j)$.
    Define the map $\varphi_{\alpha(X)}$ from $E$ into $V \cup \{\infty\}$ by \\
\begin{equation*}
\varphi_{\alpha(X)}(e):=\left\{
\begin{array}{rl}
v_j& if\, e=e_j,\, for\, some\, 1\leq j\leq k\\
\infty& if\, e\notin \{e_1, e_2,\dots,e_k\}.
\end{array}\right.
\end{equation*}
Clearly, $\varphi_{\alpha(X)}$ is a labeling of $G$ with respect to $\alpha(X)$ such that $r(X)\geq\mathrm{Hei}(\varphi_{\alpha(X)})=k$.
We claim that $r(X)=\mathrm{Hei}(\varphi_{\alpha(X)})=k$. On the contrary, suppose $r(X)>\mathrm{Hei}(\varphi_{\alpha(X)})$.
Then there exists a labeling $\bar{\varphi}_{\alpha(X)}$ with respect to $\alpha(X)$ such that $m:=r(X)=\mathrm{Hei}(\bar{\varphi}_{\alpha(X)})$.
Let $\{e\in E|\, \bar{\varphi}_{\alpha(X)} (e)\neq \infty\}=\{e_1,e_2,\dots, e_m\}$ and $\{\bar{\varphi}_{\alpha(X)} (e_1), \bar{\varphi}_{\alpha(X)} (e_2),\dots, \bar{\varphi}_{\alpha(X)} (e_m)\}=\{v_1,v_2,\dots,v_t\}$, for some positive integer $t$.
For every $1\leq i\leq t$, we set $F_i:=\{j\in \mathbb{N}|\, 1\leq j\leq m,\,  \bar{\varphi}_{\alpha(X)} (e_j)=v_i\}$.
Then it is easy to see that $\{(v_i,j)|\,1\leq i\leq t,\, 1\leq j \leq |F_i|\}$ is an independent set of $X$. It is a contradiction with this fact that $r_\mathcal{A}(X)=k$.
Hence  $M=M[\mathcal{A}]$ is a transversal matroid with rank function $r=r_\mathcal{A}$. Consider a fixed orientation $\mathcal{O}_1$ on $G$.
 With respect to this orientation, we define $\varphi^\prime_{\alpha(S)}$ from $E$ into $V \cup \{\infty\}$ by $\varphi^\prime_{\alpha(S)}(e):=t(\mathcal{O}_1(e))$.
 Clearly, $\varphi^\prime_{\alpha(S)}$ is a labeling of $G$ with respect to $\alpha(S)$.
 Thus $M$ is a matroid of rank $|E|$. Now we show that $M$ is identically self dual. For this goal, suppose $B^\prime=\{b_e|\, e\in E\}$ is a basis of $M$. Then for every $e\in E$, there exist $v_e\in V$ and $1\leq i_e \leq \mathrm{deg}_G(v_e)$ so that $b_e=(v_e, i_e)$. Also there exists a permutation $\tau$ on the set $E$ such that for every $e\in E$, we have $b_{\tau(e)}=(v_{\tau(e)}, i_{\tau(e)})\in A(e)$.
  We can consider an orientation $\mathcal{O}_2$ of $G$ such that for every $e\in E$, the vertex $v_{\tau(e)}$ is the head of $\mathcal{O}_2(e)$. Now using the orientation $\mathcal{O}_2$, we define $\varphi_{\alpha(S\backslash B^\prime)}$ from $E$ into $V \cup \{\infty\}$ by $\varphi_{\alpha(S\backslash B^\prime)}(e):=t(\mathcal{O}_2(e))$. It is easy to see that  $\varphi_{\alpha(S\backslash B^\prime)}$ is a labeling of $G$ with respect to $\alpha(S\backslash B^\prime)$.
  Therefore $r(S\backslash B^\prime)=|E|$. Also as the some of degrees of vertices of $G$ is $2|E|$, we have $|S\backslash B^\prime|=|B^\prime|$. Hence $S\backslash B^\prime$ is a basis of $M$ and it completes the proof. \\
 \textbf{ b)} It is clear that $M\backslash X$ is transversal. Since $M$ is identically self dual, we have $(M\backslash X)^\ast=M/X$. Thus it suffices to show that $M/X$ is transversal. Let $E_1$ be the edge set of $G[V\backslash W]$, where $W$ is a subset of $V$ such that $X=S(W)$. Set $E_2:=E\backslash E_1$. Now we show that $r(X)=|E_2|$. For this goal, we consider an orientation $\mathcal{O}$ on $G$ such that for every $e\in E_2$, the tail of $e$ is in $W$. We define the map $\varphi_{\alpha(X)} $ from $E$ into $V\cup\{\infty\}$ by \\
 \begin{equation*}
  \varphi_{\alpha(X)}(e):=\left\{
\begin{array}{rl}
t(\mathcal{O}(e))& if\, e\in E_2\\
\infty& if\, e\notin E_2.
\end{array}\right.
  \end{equation*}
  Obviously, $\varphi_{\alpha(X)}$ is a labeling of $G$ with respect to $\alpha(X)$. Hence $r(X)=\mathrm{Hei}(\varphi_{\alpha(X)})=|E_2|$. Thus as $r(M)=|E|$, \cite[Proposition 3.1.6]{oxley} implies that $M/X$ is a matroid of rank $|E_1|$. Now we consider the set system $\mathcal{A^\prime}:=(A(e),\, e\in E_1)$.
  \begin{lemma}\label{contract}
  $M/X=M[\mathcal{A}^\prime]$.
  \end{lemma}

  \begin{proof}
  If we consider an orientation $\mathcal{O}_1$ of $G[V\backslash W]$, then it is easy to see that $B_{\mathcal{O}_1}=\{(v,i)\in (V\backslash W)\times \mathbb{N}|\,1\leq i\leq \mathrm{deg}^+_{{\mathcal{O}_1}}(v)\}$ is a basis of $M[\mathcal{A}^\prime]$.
   Thus $M[\mathcal{A}^\prime]$ is a matroid of rank $|E_1|$. Now let $B=\{b_e|\, e\in E_1\}$ be a basis of $M/X$.
   Then for every $e\in E_1$, there exist $v_e\in V\backslash W$ and $1\leq i_e \leq \mathrm{deg}_G(v_e)$ such that $b_e=(v_e, i_e)$.
   Since $B$ is a basis of $M/X$, using \cite[Proposition 3.1.7]{oxley}, we deduce that for some basis $B_X$ of $X$, the set $B\cup B_X$ is a basis of $M=M[\mathcal{A}]$.
Thus for some permutation $\tau$ on the set $E_1$, for every $e\in E_1$, we have $b_{\tau(e)}=(v_{\tau(e)}, i_{\tau(e)})\in A(e)$.
     Hence $B$ is a basis of $M[\mathcal{A}^\prime]$.
      conversely, suppose $B^\prime=\{b^\prime_e|\, e \in E_1\}$ is a basis of $M[\mathcal{A}^\prime]$ such that $b^\prime_e \in A(e)$, for every $e\in E_1$.
      Then for each $e\in E_1$ with ends $v$ and $w$, there exist $v^\prime_e\in \{v,w\}$ and $1\leq i^\prime_e \leq \mathrm{deg}_G(v^\prime_e)$ such that $b^\prime_e=(v^\prime_e, i^\prime_e)$.
      Now we consider an orientation $\mathcal{O}_2$ of $G$ such that for every $e\in E_1$ (resp. $e\in E_2$), $t(\mathcal{O}_2(e))=v^\prime_e$ (resp. $t(\mathcal{O}_2(e))\in W$).
       Then it is clear that $\{b^\prime_e=(t(\mathcal{O}_2(e)), i^\prime_e)|\, e\in E_1\}\cup\{(v,i)\in W\times \mathbb{N}|\, 1\leq i\leq \mathrm{deg}^+_{\mathcal{O}_2}(v)\}$ is a basis of $M$. Hence using \cite[Proposition 3.1.7]{oxley}, $B^\prime$ is a basis of $M/X$.
       Therefore $M/X=M[\mathcal{A}^\prime]$.
  \end{proof}
  This completes the proof.\qed\\

\noindent\textbf{Proof of Theorem \ref{graphical}: }
 Since the number of bases of $\mathrm{TM}(G,W)$ is equal to the number of bases of $\mathrm{TM}(G,W)^\ast$, it suffices to show that the number of bases of $\mathrm{TM}(G,W)^\ast$ is equal to
 $$\bar{b}:=\sum_{(a_1,a_2,\dots,a_n)\in D^+(G[V\backslash W])}\prod_{i=1}^n\binom{\mathrm{deg}_G(v_i)}{a_i}.$$
 Let $E_1$ be the edge set of $G[V\backslash W]$. For every $e\in E_1$ with ends $v$ and $w$, suppose $A(e):=\{(x,i)|\,x\in \{v,w\},\, 1\leq i\leq \mathrm{deg}_G(x)\}$
 and $\mathcal{A}=(A(e),\,e\in E_1)$. Then by Lemma \ref{contract}, $\mathrm{TM}(G,W)^\ast=M[\mathcal{A}]$.
 Suppose $\mathcal{O}$ is an orientation of $G[V\backslash W]$.
  For each $v\in V\backslash W$, assume that $\mathcal{O}(e_1), \mathcal{O}(e_2), \dots, \mathcal{O}(e_{d(v)})$ are precisely the arcs of $\mathcal{O}(G[V\backslash W])$ with tail $v$, where $d(v):=\mathrm{deg}_\mathcal{O}^+(v)$.
 If for every $v\in V\backslash W$, we choose $d(v)$ distinct elements $(v,i_1),(v,i_2),\dots,(v,i_{d(v)})$ from the set $\{(v,i)|\,1\leq i\leq \mathrm{deg}_G(v)\}$,
  then it is clear that for every $1\leq j\leq d(v)$,
  we have $(v,i_j)\in A(e_j)$ and
   $B_{\mathcal{O}}:=\{(v,i_j)\in (V\backslash W)\times \mathbb{N}|\, 1\leq j\leq d(v)\}$
    is a basis of $M[\mathcal{A}]$.
  Hence with respect to the  orientation $\mathcal{O}$, the matroid $M[\mathcal{A}]$ has $\prod_{i=1}^n\binom{\mathrm{deg}_G(v_i)}{\mathrm{deg}_\mathcal{O}^+(v_i)}$ bases of the form $B_\mathcal{O}$.
   We call a basis of the form $B_\mathcal{O}$, an $\mathcal{O}$-basis of $M[\mathcal{A}]$.
   Now suppose $B=\{b_e|\,e\in E_1\}$ is a basis of $M[\mathcal{A}]$ such that for every $e\in E_1$, we have $b_e=(v_e,i_e)\in A(e)$.
   If we consider an orientation $\mathcal{O}^\prime$ on $G[V\backslash W]$ so that for each $e\in E_1$, the vertex $v_e$ is the tail of $\mathcal{O}^\prime(e)$, then $B$ is an $\mathcal{O}$-basis of $M[\mathcal{A}]$. Therefore the number of basis of $M[\mathcal{A}]$ is equal to $\bar{b}$. This completes the proof.\qed

  We end this section by the following example.

\begin{center}
 
  \begin{tabular}{|c|c|c|c|c||c|c|c|c|c|}\hline
 N& $a$&$|a|$&$\bar{a}$ &$\tilde{a}$&N& $a$&$|a|$&$\bar{a}$ &$\tilde{a}$\\ \hline
 1& (0,1,2,3)& 1& 9 &500 & 20& (2,0,1,3) & 1 & 9&500 \\ \hline
 2& (0,1,3,2)&1&9& 500&21& (2,0,2,2) & 2 &27&1000\\ \hline
 3& (0,2,1,3)& 1&9&500 &22& (2,0,3,1) &1  &9&500\\ \hline
 4& (0,2,2,2)&2&27& 1000&23& (2,1,0,3) & 1 &9&500\\ \hline
 5& (0,2,3,1)& 1& 9&500&24& (2,1,1,2) & 4 &81&2500\\ \hline
 6& (0,3,1,2)&1&9&500&25&(2,1,2,1)  & 4 &81&2500\\ \hline
 7& (0,3,2,1)& 1& 9&500&26& (2,1,3,0) & 1 &9&500\\ \hline
 8& (1,0,2,3)&1&9&500& 27& (2,2,0,2) & 2 &27&1000\\ \hline
 9& (1,0,3,2)& 1& 9&500& 28& (2,2,1,1) & 4 &81&2500\\ \hline
 10& (1,1,1,3)&2&27&1250& 29&(2,2,2,0)  & 2 &27&1000\\ \hline
 11& (1,1,2,2)&4&81&2500&30& (2,3,0,1) & 1 &9&500\\ \hline
 12& (1,1,3,1)& 2&27&1250&31&(2,3,1,0)  & 1 &9&500\\ \hline
 13& (1,2,0,3)& 1& 9& 500&32& (3,0,1,2) & 1 &9&500\\ \hline
 14& (1,2,1,2)& 4& 81&2500& 33& (3,0,2,1) & 1 &9&500\\ \hline
 15& (1,2,2,1)& 4& 81&2500& 34& (3,1,0,2) & 1 &9&500\\ \hline
 16& (1,2,3,0)& 1& 9&500& 35& (3,1,1,1) & 2 &27&1250\\ \hline
 17& (1,3,0,2)& 1& 9&500& 36& (3,1,2,0) & 1 &9&500\\ \hline
 18& (1,3,1,1)& 2& 27&1250& 37& (3,2,0,1) & 1 &9&500\\ \hline
 19& (1,3,2,0)& 1& 9&500& 38& (3,2,1,0) & 1 &9&500\\ \hline
 \end{tabular}\\
Table 1: list of elements of $D^+(G)$
\end{center}
 \begin{example}\label{exam}
\textbf{ a)} Assume that $G$ is a complete graph with vertex set $V_G=\{v_1,v_2,v_3,v_4\}$. For every $a=(a_1,a_2,a_3,a_4)\in D^+(G)$, we set $\bar{a}:=\prod_{i=1}^4\binom{\mathrm{deg}_G(v_i)}{a_i}$. There exist $64$ orientations on $G$. It is easy to see that $|D^+(G)|=38$ and the elements of $D^+(G)$ are precisely as in Table 1. Thus by Theorem \ref{graphical}, the number of bases of $\mathrm{TM}(G)$ is equal to $918$.\\
\textbf{ b)} Suppose $H$ is a complete graph with vertex set  $V_H=\{v_1,v_2,v_3,v_4,v_5, v_6\}$. Also let $W=\{v_5,v_6\}$. Then it is clear that $D^+(H[V\backslash W])=D^+(G)$. For every $a=(a_1,a_2,a_3,a_4)\in D^+(G)$, we set $\tilde{a}:=\prod_{i=1}^4\binom{\mathrm{deg}_H(v_i)}{a_i}$. Hence using Theorem \ref{graphical} and Table 1, the number of bases of $\mathrm{TM}(H,W)$ is equal to $36000$.
  \end{example}
\section*{Acknowledgements}
This research was supported in part
by a grant  from School of Mathematics, Institute for Research in Fundamental Sciences (IPM).


\end{document}